\documentclass[12pt, reqno]{amsart}
\usepackage{amsfonts}
\usepackage[numbers,sort&compress]{natbib}
\usepackage{ifthen}
\usepackage{amsmath}
\usepackage{graphicx}
\usepackage{amscd,amssymb,amsthm}
\newcommand{\RM}{\operatorname{Re}}

\nonstopmode \numberwithin{equation}{section}
 \numberwithin{figure}{section}
\newtheorem{definition}{Definition}[section]
\newtheorem{theorem}{Theorem}[section]
 \newtheorem{corollary}{Corollary}[section]

\newtheorem{lemma}{Lemma}[section]

\newtheorem{remark}{Remark}[section]
\begin{document}

\title[Bessel-Struve kernel function]{One some differential subordination involving the Bessel-Struve kernel function}

\author[Saiful R. Mondal]{Saiful R. Mondal}
\address{ Department of Mathematics,
King Faisal University, Ahsaa 31982, Saudi Arabia }
\email{smondal@kfu.edu.sa}
\author[Mohammed Al-Dhuain]{Mohammed Al-Dhuain}
\address{Department of Mathematics,
King Faisal University, Ahsaa , Saudi Arabia}
\email{albhishi@hotmail.com.com}
\thanks{$\ast$ First author thanks the  Deanship of Scientific Research at King Faisal University for funding this work under project number 150244.}

\maketitle

\begin{abstract}
In this article we study the inclusion properties of the Bessel-Struve kernel functions in the Janowski class. In particular, we find the conditions for which  the Bessel-Struve kernel functions maps the unit disk to right half plane. Some open problem in this aspect are also given. The third order differential subordination involving the Bessel-Struve kernel  is also considered.  The  results are derived by defining suitable classes of admissible functions. One of the recurrence relation of the Bessel-Struve kernel play an important role to derive the main results.
\end{abstract}

\section{Introduction and Preliminaries}
\subsection{Bessel-Struve Kernel functions}
Consider the Bessel-Struve kernel function $\mathtt{S}_{\alpha, \lambda}$ defined on  the unit disk ${\Delta}=\{z: |z|<1\}$ as
\begin{align}
\mathtt{S}_{\alpha, \lambda}( z):= j_\alpha(i \lambda z)- i h_\alpha(i \lambda z), \quad \alpha> -\frac{1}{2}, \quad \lambda \in \mathbb{C},
\end{align}
where, $j_\alpha(z):= 2^\alpha z^{-\alpha} \Gamma{(\alpha+1)} \mathtt{J}_\alpha(z)$ and $h_\alpha(z):= 2^\alpha z^{-\alpha} \Gamma{(\alpha+1)} \mathtt{H}_\alpha(z)$ are respectively known as the normalized Bessel functions and the normalized Struve functions of first kind of index $\alpha$. The Bessel-Struve transformation and Bessel-Struve kernel functions  are appeared in many article. In \cite{Hamem}, Hamem et. al. studies an analogue of the Cowling-Price theorem for the Bessel-Struve transform defined on real domain and also provide Hardyâ's type theorem associated with this transform. The Bessel-Struve intertwining operator on $\mathbb{C}$ is considered in \cite{Gasmi}. The fock space of the Bessel-Struve kernel functions is discussed in \cite{Gasmi-Soltani}.

The kernel $z \mapsto \mathtt{S}_{\alpha, \lambda}( z)$, $\lambda \in \mathbb{C}$ is the unique solution of the initial value problem
\begin{align}\label{eqn:Bessel-struve-opt}
\mathcal{L}_\alpha u(z)= \lambda^2 u(z), \quad u(0)=1, u'(0)=\frac{\lambda \Gamma{(\alpha+1)}}{\sqrt{\pi}\Gamma{(\alpha+\frac{3}{2}})}.
\end{align}
Here $\mathcal{L}_\alpha$ , $\alpha >-1/2$ is the Bessel-Struve operator given by
\begin{align}\label{eqn:Bessel-struve-opt-1}
\mathcal{L}_\alpha(u(z)):= \frac{d^2u}{dz^2}(z)+\frac{2\alpha+1}{z}\left(\frac{du}{dz}(z)-\frac{du}{dz}(0)\right).
\end{align}

Now, the Bessel functions and the Struve functions of order $\alpha$ respectively  have the power series
\begin{align*}
J_\alpha(z)= \sum_{n=0}^\infty \frac{(-1)^n\left(\frac{z}{2}\right)^{2n+\alpha}}{n! \Gamma{(\alpha+n+1)} } \quad \text{and} \quad
\mathtt{H}_\alpha(z):=  \sum_{n=0}^\infty \frac{(-1)^n \left(\frac{z}{2}\right)^{2n+\alpha+1}}{\mathrm{\Gamma}{\left(n +\alpha+\frac{3}{2}\right)} \mathrm{\Gamma}{\left(n+\frac{3}{2}\right)}}.
\end{align*}
This implies that  $\mathtt{S}_{\alpha, \lambda}$   possesses the power series
\begin{align}\label{eqn:B-S-power}
\mathtt{S}_{\alpha, \lambda}( z):= \sum_{n=0}^\infty \frac{ {\Gamma(\alpha+1)}\Gamma{\left(\frac{n+1}{2}\right)}}{\sqrt{\pi} n! \Gamma\left(\frac{n}{2}+\alpha+1\right)} \lambda^n z^n.
\end{align}
The kernel $\mathtt{S}_{\alpha, \lambda}$ also have the integral representation
\begin{align}\label{kernel-intgra-rep}
\mathtt{S}_{\alpha, \lambda}( z):=\frac{ 2 \Gamma{(\alpha+1)}}{\sqrt{\pi}\Gamma{\left(\alpha+\frac{1}{2}\right)}}\int_{0}^1 (1-t^2)^{\alpha-\frac{1}{2}} e^{\lambda zt} dt.
\end{align}
Now from \eqref{eqn:Bessel-struve-opt} and \eqref{eqn:Bessel-struve-opt-1} it is evident that $\mathtt{S}_\alpha$ satisfy the differential equation
\begin{align}\label{eqn:kumar-hypr-ode}
z^2 \mathtt{U}''( z)+(2\alpha+1)z\mathtt{U}'(z)-z\lambda^2\mathtt{U}(z)= z M ,
\end{align}
where $M=2 \lambda\Gamma(\alpha+1)\left(\sqrt{\pi}\;\Gamma(\alpha +\frac{1}{2})\right)^{-1}$.

It can be shown that  $\mathtt{S}_{\alpha, 1}=\mathtt{S}_{\alpha}$  satisfy the recurrence relation
\begin{align}\label{eqn:Bessel-struve-rec}
z \mathtt{S}'_\alpha(z)= 2\alpha \mathtt{S}_{\alpha-1}(z)-2 \alpha \mathtt{S}_\alpha(z).
\end{align}

\subsection{Differential subordinations}
 In this sections a details introduction about the classes of univalent functions theory,  admissible functions  and fundamental results about the differential subordination of different orders are given. The  differential subordinations and its application are  mainly encompassed  the first order and second order differential subordination till the introduction of the third order differential technique by  Antonino\ and\  Miller \cite{Antonino-Miller}.

Let $\mathcal{A}$ denote the class of analytic functions $f$ defined in the open unit disk $\Delta$ normalized by the conditions
 $f(0) = 0 = f'(0)-1$ and have the form
\begin{align}
f(z)= z+ \sum_{n=1}^\infty a_{n+1} z^{n+1}.
\end{align}
If $f$ and $g$ are  analytic
in $\Delta$, then $f$ is subordinate to $g$, written  $f(z) \prec g(z)$,  if there is an analytic self-map
$w$ of $\Delta$ satisfying $w(0)=0$  and $f = g \circ w$.
For $-1 \leq B < A \leq  1$, let $\mathcal{P}[A,B]$ be the class consisting of normalized analytic functions $p(z)= 1+ c_{1}z + \cdots$ in $\Delta$ satisfying
\[ p(z) \prec \frac{1+Az}{1+Bz}.\]
For instance, if $0 \leq \beta <1$, then $ \mathcal{P}[1-2 \beta, -1]$ is the class of functions $p(z)= 1+ c_{1}z + \cdots$ satisfying $ \RM p(z) > \beta$ in $\Delta$.

The class $\mathcal{S}^\ast[A, B]$ of Janowski starlike functions \cite{Janowski} consists of $ f \in \mathcal{A}$ satisfying
\[\frac{ z f'(z)}{f(z)}\in \mathcal{P}[A,B].\]
For $0 \leq \beta <1$, $ \mathcal{S}^\ast[1-2 \beta, -1]:=  \mathcal{S}^\ast(\beta)$ is the usual class of starlike functions of order $\beta$;
$\mathcal{S}^\ast[1- \beta, 0]:=  \mathcal{S}^\ast_{\beta} = \{f \in \mathcal{A} :
| z f'(z)/f(z) -1| < 1- \beta \}$, and $\mathcal{S}^\ast[\beta, - \beta]:=\mathcal{S}^\ast[\beta]= \{f \in \mathcal{A} :
| z f'(z)/f(z) -1| < \beta | z f'(z)/f(z) +1|\}$. These classes have been studied, for example, in
\cite{Ali-seeni-ijmms,Ali-Chandra-aml}. A function $f \in \mathcal{A}$ is said to be close-to-convex of order $\beta$ \cite{Miller-Mocanu-book,Goodman-book} if $\RM\left( zf'(z)/{g(z)}\right)> \beta$ for some $g \in \mathcal{S}^\ast :=\mathcal{S}^\ast(0)$.

Let $\mathcal{H}(\Delta)$ be the class of functions which are analytic in $\Delta$. For $n \in \mathbb{N}:= \{1, 2, 3, \ldots\}$ and $a \in \mathbb{C}$, consider the class $\mathcal{H}[a,n]$ defined as
 \begin{align}
\mathcal{H}[a,n]:=\{f \in \mathcal{H}(\Delta): f(z)= a + a_{n} z^{n} + a_{n+1} z^{n+1}+\ldots\},
\end{align}
and suppose that $\mathcal{H}_0=\mathcal{H}[0,1]$.

The theory of the differential subordination in $\Delta$ was introduced by Miller and Mocanu \cite{Miller-Mocanu-book} and
subsequently many researcher either apply this  concept to study the geometric properties of analytic functions defined on
$\Delta$ and developed or reproduced  several other theory for  subclasses of the univalent functions theory
(See \cite{Ali-2,Ali-3,Ali-4,Ali-5,Ali-6,Ali-1} and references their in).

Following result is important to show that the Bessel-Struve kernel have a relation with Janowski class.
\begin{lemma}\cite{miller-dif-sub,Miller-Mocanu-book}\label{lem:miller-mocanu-1}
 Let $\Omega \subset \mathbb{C}$, and $\Psi : \mathbb{C}^3 \times \Delta \to \mathbb{C}$ satisfy
 \[
 \Psi( i \rho , \sigma, \mu + i \nu; z) \not \in \Omega
 \]
 whenever $z \in \Delta$, $\rho$ real, $\sigma \leq -(1+\rho^2)/2$ and $ \sigma+\mu \leq 0$. If $p$ is analytic in $\Delta$ with $p(0)=1$, and
 $\Psi( p(z), z p'(z), z^2 p''(z); z) \in \Omega$ for $ z \in \Delta$, then $\RM  p(z) > 0$ in $\Delta$.
\end{lemma}
Theory of the second-order differential subordinations in $\Delta$ have extended to third-order differential subordinations
by  Antonino\ and\  Miller \cite{Antonino-Miller}. They determined properties of analytic functions $p$ in $\Delta$ that satisfy
the following third order differential subordination:
\begin{align*}
\{\psi( p(z), z p'(z), z^2 p''(z), z^3 p'''(z); z)\} \subset \Omega,
\end{align*}
where, $z \in \Delta$,  $\Omega$ is any set in $\mathbb{C}$, and $\psi: \mathbb{C}^4 \times \Delta \to  \mathbb{C}$.

Next we write down few important definition and results from \cite{Antonino-Miller} which are required in sequel.
\begin{definition}{\rm\cite[p. 440]{Antonino-Miller}}
Let $\psi: \mathbb{C}^4 \times \Delta \to  \mathbb{C}$ and $h$ be univalent in $\Delta$. If $p$ is analytic in $\Delta$ and satisfies the third-order differential subordination
\begin{align}\label{eqn:def-1-miller}
\psi(p(z), z p'(z), z^2 p''(z), z^3 p'''(z); z)\prec h(z)
\end{align}
then $p$ is called a solution of the differential subordination. A univalent function $q$ is
called a dominant of the solutions of the differential subordination or more simply a
dominant if $p\prec q$ for all $p$ satisfying \eqref{eqn:def-1-miller}. A dominant $\tilde{q}$ that satisfies
$\tilde{q} \prec q$ for all
dominant of \eqref{eqn:def-1-miller} is called the best dominant of \eqref{eqn:def-1-miller}. Note that the best dominant is unique up to a rotation of $\Delta$.
\end{definition}

\begin{definition}{\rm\cite[p. 441]{Antonino-Miller}}
Let $\mathcal{Q}$ denote the set of functions $Q$ that are analytic and univalent on the set
$\overline{\Delta}\setminus E(q) $, where
\[ E(q)= \{ \zeta \in \partial\Delta : \lim_{z \to \zeta} q(z)= \infty\},\]
and are such that
\[ \min|q'(\zeta)|=\rho>0\]
for $\zeta \in \overline{\Delta}\setminus E(q) $. Further, let the subclass of $\mathcal{Q}$  for which $q(0)=a$ be denoted by
$\mathcal{Q}(0)=\mathcal{Q}_0$.
\end{definition}

\begin{definition}{\rm\cite[p. 449]{Antonino-Miller}}\label{def-miller-anton}
Let $\Omega$ be any set in $\mathbb{C}$, $q \in \mathcal{Q}$ and $n \in \mathbb{N}\setminus\{1\}.$
The class of admissible function $\Psi_n[\Omega, q]$ consists of those functions  $\psi: \mathbb{C}^4 \times \Delta \to  \mathbb{C}$ that satisfy the admissibility condition:
$
\psi( r, s, t, u; z) \notin \Omega,
$
whenever
\[r= q(\zeta), \quad s=m \zeta q'(\zeta), \quad \RM \left(\frac{t}{s}+1\right) \geq m \RM\left(\frac{\zeta q''(\zeta)}{q'(\zeta)}+1\right)\]
and
\[ \RM\left(\frac{u}{s}\right) \geq m^2 \RM\left(\frac{\zeta^2 q'''(\zeta)}{q'(\zeta)}+1\right)\]
where $z \in \Delta$, $\zeta \in \overline{\Delta}\setminus E(q),$ and $m \geq n$.
\end{definition}

\begin{theorem}{\rm\cite[p. 449]{Antonino-Miller}}\label{thm-miller-anton}
Let $p \in \mathcal{H}[a, n]$ with $n \geq 2$. Also let $q \in \mathcal{Q}(a)$ and satisfy the following conditions:
\begin{align}\label{eqn:Antonino-Miller-thm-1}
\RM\left(\frac{\zeta^2 q'''(\zeta)}{q'(\zeta)}\right) \geq 0 \quad \text{and} \quad \left|\frac{z p'(z)}{q'(\zeta)}\right|\leq m,
\end{align}
where $z \in \Delta$, $\zeta \in \overline{\Delta}\setminus E(q),$ and $m \geq n$. If $\Omega$ is any set in $\mathbb{C}$,
$\psi \in \Psi_n[\Omega, q]$ and
\begin{align*}
\psi(p(z), z p'(z), z^2 p''(z), z^3 p'''(z); z)\in \Omega,
\end{align*}
then $p(z) \prec q(z)$.
\end{theorem}

Recently, Tang and Deniz \cite{Tang-Deniz} applied the above third-order  differential subordination concept
to an operator involving the generalized Bessel functions by considering suitable class of admissible functions.

In section \ref{section-2}, we obtain the sufficient condition for which the Bessel-Struve kernel functions have relation with  the functions of the Janowski class.  The classes of  admissible functions applicable for the Besse-Struve kernel function in $\Delta$ are introduced in Section \ref{section-3}  and then results related to the third order differential subordinations involving this function are derived.

\section{Inclusion of Bessel-Struve kernel  in the Janowski Class}\label{section-2}

\begin{theorem}\label{thm:-janw-cc}
Let $-1 \leq B \leq 3- 2 \sqrt{2} \approx 0.171573$. Suppose
 $B < A \leq 1$, and $\lambda$, $\alpha  \in \mathbb{R}$ satisfy
 \begin{align}\label{eqn:thm-janw-cc-0}
\alpha \geq \max \left\{0,  \tfrac{\left|\lambda \right|}{2}\left|\tfrac{\lambda (1+A)(1+B)+M(1+B)^2}{A-B}\right|\right\}.
\end{align}
Further let $A$, $B$, $\alpha$ and $\lambda$ satisfy either the inequality
\begin{align}\label{eqn:thm-janw-cc-2}\nonumber
4\alpha ^2 -&\frac{\lambda}{A-B} \left|4\alpha  \big(\lambda(A+B)+2MB\big)+\frac{(1+B)^2}{1-B}\big(\lambda(1+A)+ M(1+B)\big)\right|
+2 \alpha \frac{1+B}{1-B}\\
&\geq \frac{\lambda^2(1-B^2)(\lambda(1-A)+ M (1-B))(\lambda(1+A)+ M(1+B))}{(A-B)^2},
\end{align}
whenever
\begin{align}\label{eqn:thm-janw-cc-1}
\left|4\alpha  \big(\lambda (A+B)+2MB\big)(1-B)+(1+B)^2\big(\lambda(1+A)+ M(1+B)\big)\right|
 \geq 2 \lambda^3(1-B)(A-B)
\end{align}
or the inequality
\begin{align}\label{eqn:thm-janw-cc-3}\nonumber
&\big(4\alpha \lambda \big(\lambda (A+B)+2MB\big)+\frac{1+B}{1-B}\big(\lambda^2 (1+A)(1+B)+\lambda M(1+B)^2\big)\big)^2\\ \nonumber
& \quad \leq 4((\lambda^2(1-AB)+\lambda M(1-B^2))^2-(\lambda^2(1-A)(1-B)+\lambda M(1-B)^2)\\&(\lambda^2(1+A)(1+B)+\lambda M(1+B)^2))\big(4\alpha ^2 +2\alpha \frac{1+B}{1-B}-\left(\tfrac{\lambda^2(1-AB)+\lambda M(1-B^2)}{A-B}\right)^2\big),
\end{align}
whenever
\begin{align}\label{eqn:thm-janw-cc-4}
\left|4\alpha  \big(\lambda (A+B)+2MB\big)(1-B)+(1+B)^2\big(\lambda(1+A)+ M(1+B)\big)\right|
< 2 \lambda^3(1-B)(A-B).
\end{align}
If $(1+B)\mathtt{S}_{\alpha, \lambda }(z) \neq (1+A)$,  then
$\mathtt{S}_{\alpha, \lambda }(z) \in \mathcal{P}[A,B]$.
\end{theorem}

\begin{proof}
Define the analytic function $p : \Delta \to  \mathbb{C}$ by
\begin{align*}
p(z) := - \frac{ (1-A) - (1-B) S_{\alpha, \lambda}(z)}{ (1+A) - (1+B) \mathtt{S}_{\alpha, \lambda }(z)},
\end{align*}
 Then
\begin{align}\label{eqn-thm-1-1}
\mathtt{S}_{\alpha, \lambda }(z)& = \frac{(1-A) + (1+A) p(z)}{(1-B) + (1+B) p(z)},\\ \label{eqn-thm-1-2}
\mathtt{S}'_{\alpha, \lambda }(z)& = \frac{ 2 (A-B) p'(z)}{((1-B) + (1+B) p(z))^2 },\\
\label{eqn-thm-1-3}
 \mathtt{S}''_{\alpha, \lambda }(z) &=\frac{2 (A-B)( (1-B) + (1+B) p(z) ) p''(z)
- 4 (1+B) (A-B) {p'}^2(z)}{( (1-B) + (1+B) p(z) )^3}.
\end{align}
Using $(\ref{eqn-thm-1-1})$--$(\ref{eqn-thm-1-3})$, the Bessel-Struve differential equation
$(\ref{eqn:kumar-hypr-ode})$ yields
\begin{align}\label{eqn:thm-1-ode} \nonumber
z^2p''(z)& - \tfrac{ 2( 1+B)}{(1-B)+ (1+B)p(z)}  (zp'(z))^2 + (2\alpha +1) zp'(z) \\
&-\bigg(\tfrac{(\lambda^2(1-A)+(1+A)p(z))((1-B)+(1+B)p(z))+\lambda M((1-B)+(1+B)p(z))^2}{2(A-B)}\bigg)z = 0.
\end{align}

With $\Omega = \{0\}$, define $\Psi(r, s, t;z)$ by
\begin{align}\label{eqn:thm-1-ode-2}\nonumber
&\Psi(r, s, t;z) := t - \tfrac{ 2( 1+B)}{(1-B)+ (1+B)r} s^2 + (2\alpha +1)s\hspace{2in}\\&\hspace{1in}-
\bigg(\tfrac{(\lambda^2(1-A)+(1+A)r)((1-B)+(1+B)r)+\lambda M((1-B)+(1+B)r)^2}{2(A-B)}\bigg)z.
\end{align}
It follows from  $(\ref{eqn:thm-1-ode})$ that
$
\Psi(p(z), z p'(z), z^2 p''(z); z)  \in \Omega.$  To show $\RM p(z) >0$ for $z \in \Delta$,
from Lemma $\ref{lem:miller-mocanu-1}$, it is sufficient to establish
$
\RM \Psi( i\rho, \sigma, \mu+ i \nu; z) < 0
$
in $\Delta$  for any real $\rho$, $\sigma \leq -(1+\rho^2)/2$, and $\sigma+\mu \leq 0$.

With $z = x+ iy \in \Delta$, it readily follows from $(\ref{eqn:thm-1-ode-2})$ that
\begin{align}\label{eqn:thm-1-re-psi}\nonumber
\RM \Psi( i\rho, \sigma, \mu+ i \nu; z)
& = \mu - \tfrac{ 2( 1-B^2)}{(1-B)^2+ (1+B)^2 \rho^2} \sigma^2 + (2\alpha+1) \sigma
 +\left(\tfrac{\lambda^2(1-AB)+\lambda M (1-B^2)}{(A-B)}\right)\rho y\\
&\hspace{.5in}-\left(\tfrac{\lambda^2((1-A)(1-B)-(1+A)(1+B)\rho^2)+\lambda M ((1-B)^2-(1+B)^2\rho^2)}{2(A-B)}\right)x.
\end{align}
Since $\sigma \leq -(1+\rho^2)/2$, and $B \in [-1, 3-2\sqrt{2}]$,
\begin{align*}
\frac{ 2( 1-B^2)}{(1-B)^2+ (1+B)^2 \rho^2} \sigma^2
 &\geq \frac{ 2( 1-B^2)}{(1-B)^2+ (1+B)^2 \rho^2} \frac{(1+\rho^2)^2}{4}
 \geq \dfrac{1+B}{2(1-B)}.
\end{align*}
 Thus
\begin{align*}
&\RM \Psi( i\rho, \sigma, \mu+ i \nu; z)\\
&\leq 2\alpha\sigma -\tfrac{\lambda ^2 ((1-A)(1-B)-(1+A)(1+B)\rho^2)+\lambda M((1-B)^2-(1+B)^2\rho^2)}{2(A-B)}x
+ \tfrac{\lambda^2(1-AB)+\lambda M (1-B^2)}{A-B}\rho y-\tfrac{1+B}{2(1-B)}\\
&\leq -\alpha (1+\rho^2) -\tfrac{\lambda ^2 ((1-A)(1-B)-(1+A)(1+B)\rho^2)+\lambda M((1-B)^2-(1+B)^2\rho^2)}{2(A-B)}x
+ \tfrac{\lambda^2(1-AB)+\lambda M (1-B^2)}{(A-B)}\rho y-\tfrac{1+B}{2(1-B)} \\
&=p_1 \rho^2+  q_1 \rho + r_1 := Q(\rho),
\end{align*}
where
\begin{align*}
p_1&= -\alpha +\frac{\lambda^2(1+A)(1+B)+\lambda M(1+B)^2}{2(A-B)}x, \hspace{2.5in} \\
q_1&= \frac{\lambda^2(1-AB)+\lambda M (1-B^2)}{(A-B)}y,\\
r_1 &= -\alpha - \frac{1+B}{2(1-B)}-\bigg(\frac{\lambda^2(1-A)(1-B)+\lambda M(1-B)^2}{2(A-B)}\bigg)x.
\end{align*}

Condition $(\ref{eqn:thm-janw-cc-0})$ shows that
\begin{align*}
p_1 &= -\alpha +\frac{\lambda^2(1+A)(1+B)+\lambda M(1+B)^2}{2(A-B)}x\\
 &< -\bigg(\alpha - \frac{|\lambda|}{2}\bigg|\frac{\lambda (1+A)(1+B)+M(1+B)^2}{A-B}\bigg|\bigg) <0.
\end{align*}

 Since $\displaystyle{ \max_{\rho \in \mathbb{R}} \{ p_1 \rho^2+ q_1\rho + r_1\}=(4 p_1 r_1 - q_1^2)/(4 p_1)}$ for $p_1 < 0$, it is clear that $Q(\rho) <0$ when
 \begin{align*}
\bigg(\frac{\lambda^2(1-AB)+\lambda M (1-B^2)}{A-B}\bigg)^2 y^2 &<  \bigg(-2\alpha +\bigg(\frac{\lambda^2(1+A)(1+B)+\lambda M (1+B)^2}{A-B}\bigg)x\bigg)\\
&\times \bigg(-2\alpha -\frac{1+B}{1-B}-\bigg(\frac{\lambda^2(1-A)(1-B)+\lambda M(1-B)^2}{A-B}\bigg)x\bigg),
 \end{align*}
 $|x|, |y| < 1$.
 As $y^2 < 1- x^2$, the above condition holds whenever
 \begin{align*}
&\bigg(\frac{\lambda^2(1-AB)+\lambda M (1-B^2)}{A-B}\bigg)^2 (1-x^2)\\ &<  \bigg(-2\alpha +\bigg(\frac{\lambda^2(1+A)(1+B)+\lambda M (1+B)^2}{A-B}\bigg)x\bigg)\\
& \hspace{.1in} \times \bigg(-2\alpha -\frac{1+B}{1-B}-\bigg(\frac{\lambda^2(1-A)(1-B)+\lambda M(1-B)^2}{A-B}\bigg)x\bigg),
 \end{align*}
 that is, when
 \begin{align}\label{eqn:thm-1-x} \nonumber
&\tfrac{(\lambda^2(1-AB)+\lambda M (1-B^2))^2-(\lambda^2(1-A)(1-B)+\lambda M (1-B)^2)(\lambda^2(1+A)(1+B)+\lambda M(1+B)^2)}{(A-B)^2}x^2\quad\quad\\   \nonumber
&\quad \quad +\tfrac{1}{A-B}\big(-4\alpha \lambda (\lambda (A+B)+2MB)) -\lambda \tfrac{(1+B)^2}{1-B}(\lambda(1+A)+M (1+B))\big)x
 \\
 &\hspace{.5in} + 4\alpha^2+2\alpha \tfrac{1+B}{1-B}-\big(\tfrac{\lambda^2(1-AB)+\lambda M (1-B^2)}{A-B}\big)^2 \geq 0.
\end{align}

To establish inequality $(\ref{eqn:thm-1-x})$, consider the polynomial $R$ given by
\begin{align*}\label{eqn:thm-1-R(x)}
R(x) := m x^2 + n x+ r,\quad |x| <1,
\end{align*}
where
\begin{align*}
m &:=\tfrac{(\lambda^2(1-AB)+\lambda M (1-B^2))^2-(\lambda^2(1-A)(1-B)+\lambda M (1-B)^2)(\lambda^2(1+A)(1+B)+\lambda M(1+B)^2)}{(A-B)^2},\\
&=\tfrac{(\lambda^2(1-AB)+\lambda M (1-B^2))^2-\lambda^2(1-B^2)(\lambda(1-A)+ M (1-B))(\lambda(1+A)+M(1+B))}{(A-B)^2},\\
n &:= \tfrac{ \lambda}{A-B}\big(-4\alpha (\lambda (A+B)+2MB)) - \tfrac{(1+B)^2}{1-B}(\lambda(1+A)+M (1+B))\big),\\
r &:= 4\alpha^2+2\alpha \tfrac{1+B}{1-B}-\big(\tfrac{\lambda^2(1-AB)+\lambda M (1-B^2)}{A-B}\big)^2.
\end{align*}
The constraint $(\ref{eqn:thm-janw-cc-1})$ yields $|n| \geq 2|m|$, and thus $R(x) \geq m +  r - |n|$. Now the inequality
 $(\ref{eqn:thm-janw-cc-2})$  implies that
\begin{align*}
 &R(x)\\&\geq  m +  r - |n|\\
 &=\tfrac{(\lambda^2(1-AB)+\lambda M (1-B^2))^2-\lambda^2(1-B^2)(\lambda(1-A)+ M (1-B))(\lambda(1+A)+M(1+B))}{(A-B)^2}+ 4\alpha^2+2\alpha \tfrac{1+B}{1-B}\\&-\big(\tfrac{\lambda^2(1-AB)+\lambda M (1-B^2)}{A-B}\big)^2
 - \tfrac{|\lambda|}{A-B}\big|\big(-4\alpha (\lambda (A+B)+2MB)) - \tfrac{(1+B)^2}{1-B}(\lambda(1+A)+M (1+B))\big)\big|\\
& =4\alpha^2 - \tfrac{|\lambda|}{A-B}\big|\big(-4\alpha (\lambda (A+B)+2MB)) - \tfrac{(1+B)^2}{1-B}(\lambda(1+A)+M (1+B))\big)\big|\\&+2\alpha\frac{1+B}{1-B}-\tfrac{(\lambda^2(1-A)(1-B)+\lambda M (1-B)^2)(\lambda^2(1+A)(1+B)+\lambda M (1+B)^2}{(A-B)^2} \geq 0.
\end{align*}

Now considers the case  of the constraint $(\ref{eqn:thm-janw-cc-4})$, which is equivalent to $|n| < 2m$. Then the minimum of $R$ occurs at $x = - n/(2m)$, and
$(\ref{eqn:thm-janw-cc-3})$ yields
\begin{align*}
R(x) \geq  \frac{4mr-n^2}{4m} \geq 0.
\end{align*}

Evidently  $\Psi$ satisfies the hypothesis of Lemma $\ref{lem:miller-mocanu-1}$, and thus $\RM\; p(z) > 0$, that is,
\[
- \frac{ (1-A) - (1-B) \mathtt{S}_{\alpha}(z)}{ (1+A) - (1+B) \mathtt{S}_{\alpha}(z)} \prec \frac{1+z}{1-z}.
\]
Hence there exists an analytic self-map  $w$ of $\Delta$ with $w(0)=0$ such that
\[
- \frac{ (1-A) - (1-B) \mathtt{S}_{\alpha}(z)}{ (1+A) - (1+B) \mathtt{S}_{\alpha}(z)} = \frac{1+w(z)}{1-w(z)},
\]
which implies that $\mathtt{S}_{\alpha}(z) \prec (1+ A z)/(1+B z).$
\end{proof}

Considering $\lambda=A=-B=1$, following result can be obtain from Theorem $\ref{thm:-janw-cc}$.
\begin{corollary}\label{cor:s-p-real}
For $\alpha \in [0, \alpha_0] \cup [3/2,\infty)$, $\RM(\mathtt{S}_\alpha(z))>0$. Here $\alpha_0=0.5$ is the positive root of the identity $4\alpha\Gamma{(\alpha+1)}=\sqrt{\pi}\Gamma{(\alpha+1/2)}.$
\end{corollary}
This result along with the recurrence relation \eqref{eqn:Bessel-struve-rec} gives that
\[\RM\left(\frac{z  \mathtt{S}'_\alpha(z)+2\alpha\mathtt{S}_\alpha(z)}{2\alpha}\right)>0.\]
In particular,  the function $z \mathtt{S}_{1/2}(z)$ is close-to-convex functions with respect to $z$, and hence it is univalent.
\begin{theorem}\label{thm:-janw-cc-2}
Let $3- 2 \sqrt{2} \leq B < A\leq1 $
  and $\lambda$, $\alpha  \in \mathbb{R}$ satisfy
 \begin{align}\label{eqn:thm-2-janw-cc-0}
\alpha \geq \max \left\{0,  \frac{|\lambda|}{2}\bigg|\frac{\lambda(1+A)(1+B)+M(1+B)^2}{A-B}\bigg|\right\}.
\end{align}
suppose $A$, $B$, $\alpha$ and $\lambda$ satisfy either the inequality
\begin{align}\label{eqn:thm-2-janw-cc-2}\nonumber
&(\alpha^2(A-B)^2-\lambda(A-B)\big|\alpha  (\lambda(A+B)+2MB)+\tfrac{4B(1-B)}{(1+B)^2}(\lambda(1+A)+ M(1+B))\big|\\ \notag
& +8\alpha\tfrac{B(1-B)(A-B)^2}{(1+B)^3}\geq  \tfrac{1}{4}\big(\lambda^2(1-A)(1-B)+\lambda M (1-B)^2\big)\big(\lambda^2(1+A)(1+B)+\lambda M(1+B)^2\big)\\&
\end{align}
whenever
\begin{align}\label{eqn:thm-2-janw-cc-1}\nonumber
&(A-B)\left|\alpha \lambda (\lambda (A+B)+2MB)+\frac{4B\lambda(1-B)}{(1+B)^2}\big(\lambda(1+A)+ M(1+B)\big)\right|\hfill\\ \notag
&  \geq \frac{\lambda^2}{2}\left|\big(\lambda(1-AB)+ M (1-B^2)\big)^2-(1-B^2)\big(\lambda(1-A)+ M(1-B)\big)\big(\lambda(1+A)+M(1+B)\big)\right|,\\
&
\end{align}
or the inequality
\begin{align}\label{eqn:thm-2-janw-cc-3}\nonumber
&\bigg(\alpha \bigg(\lambda(A+B)+2MB\bigg)+\frac{4B(1-B)}{(1+B)^2}\bigg(\lambda(1+A)+ M (1+B)\bigg)\bigg)^2\\ \notag
& \leq \big(\lambda(1-AB)+ M (1-B^2)\big)^2-(1-B^2)\big(\lambda (1-A)+ M(1-B)\big)\\&\big(\lambda(1+A)+ M (1+B)\big)\bigg(\alpha^2+8\alpha B\frac{1-B}{(1+B)^3}-\bigg(\frac{\lambda^2(1-AB)+\lambda M(1-B^2)}{2(A-B)}\bigg)^2\bigg)
\end{align}
whenever
\begin{align}\label{eqn:thm-2-janw-cc-4}\nonumber
&(A-B)\left|\alpha \lambda (\lambda (A+B)+2MB)+\frac{4B\lambda(1-B)}{(1+B)^2}\big(\lambda(1+A)+ M(1+B)\big)\right|\hfill\\ \notag
&  \leq \frac{\lambda^2}{2}\left|\big(\lambda(1-AB)+ M (1-B^2)\big)^2-(1-B^2)\big(\lambda(1-A)+ M(1-B)\big)\big(\lambda(1+A)+M(1+B)\big)\right|,\\
&
\end{align}
If $(1+B)\mathtt{S}_{\alpha} (z) \neq (1+A)$,  then
$\mathtt{S}_{\alpha} (z) \in \mathcal{P}[A,B]$.
\end{theorem}
\begin{proof}
Proceeding similarly as in the proof of Theorem $\ref{thm:-janw-cc}$, consider
$\RM \Psi( i\rho, \sigma, \mu+ i \nu; z)$ as given in $(\ref{eqn:thm-1-re-psi})$.
For $\sigma \leq -(1+\rho^2)/2$, $\rho \in \mathbb{R}$, and $B \geq 3- 2 \sqrt{2}$,
\begin{align*}
\frac{ 2( 1-B^2)}{(1-B)^2+ (1+B)^2 \rho^2} \sigma^2
 &\geq \frac{ 2( 1-B^2)}{(1-B)^2+ (1+B)^2 \rho^2} \frac{(1+\rho^2)^2}{4}
\geq \frac{8 B (1-B)}{(1+B)^3}.
\end{align*}
With $z = x+iy \in \Delta$, and $\mu+\sigma<0$, it follows that
\begin{align*}
&\RM\Psi( i\rho, \sigma, \mu+ i \nu; z)\\
&\leq -\alpha(1+\rho^2)-8\tfrac{B(1-B)}{(1+B)^3}+\left(\tfrac{\lambda^2(1-AB)+\lambda M (1-B^2)}{A-B} \right)\rho y\\&-\tfrac{1}{2(A-B)}\left(\lambda^2((1-A)(1-B)-(1+A)(1+B)\rho^2)+\lambda M ((1-B)^2-(1+B)^2\rho^2)\right)x\\
& = p_2 \rho^2 + q_2 \rho + r_2: = Q_1(\rho),
\end{align*}
where
\begin{align*}
p_2&= -\alpha +\tfrac{\lambda^2(1+A)(1+B)+\lambda M (1+B)^2}{2(A-B)}x, \hspace{2.5in} \\
q_2&=\tfrac{\lambda^2(1-AB)+\lambda M (1-B^2)}{A-B} y,\\
r_2 &= -\alpha - 8B \tfrac{1-B}{(1+B)^3}-\bigg(\tfrac{\lambda^2(1-A)(1-B)+\lambda M (1-B)^2}{2(A-B)}\bigg)x.
\end{align*}

In the proof of Theorem $\ref{thm:-janw-cc}$, it can be  observed that the constraint
$(\ref{eqn:thm-2-janw-cc-1})$ implies  $p_2 <0$. Thus $Q_1(\rho) <0$ for all $\rho \in \mathbb{R}$ provided $q_{2}^{2} \leq 4p_{2}r_2$, that is,
\begin{align*}
&\left(\tfrac{\lambda^2(1-AB)+\lambda M (1-B^2)}{A-B}\right)^2 y^2\leq   \left(-2\alpha +\tfrac{\lambda^2(1+A)(1+B)+\lambda M(1+B)^2}{A-B}x\right)\\
& \hspace{2in} \times \left(-2\alpha - 16B \tfrac{1-B}{(1+B)^3}-\tfrac{\lambda^2(1-A)(1-B)+\lambda M(1-B)^2}{A-B}x\right),
 \end{align*}
 $|x|, |y|<1$.  With $y^2 < 1- x^2$, it is sufficient to show
 \begin{align*}
&\left(\tfrac{\lambda^2(1-AB)+\lambda M (1-B^2)}{A-B}\right)^2 (1-x^2)\leq   \left(-2\alpha +\tfrac{\lambda^2(1+A)(1+B)+\lambda M(1+B)^2}{A-B}x\right)\hspace{2in}\\
& \hspace{2in} \times \left(-2\alpha - 16 \tfrac{B(1-B)}{(1+B)^3}-\tfrac{\lambda^2(1-A)(1-B)+\lambda M(1-B)^2}{A-B}x\right),
 \end{align*}
for $|x|<1$. The above inequality is equivalent to showing
\begin{align}\label{eqn:thm-2-x}
R_1(x) := m_1 x^2 + n_1 x+ r_1 \geq 0,
\end{align}
where
\begin{align*}
m_1  &:=\frac{-1}{(A-B)^2}\big(\big(\lambda^2(1-A)(1-B)+\lambda M (1-B)^2\big)\big(\lambda^2(1+A)(1+B)+\lambda M(1+B)^2\big)\big)\\
&\quad+\bigg(\frac{\lambda^2(1-AB)+\lambda M (1-B^2)}{(A-B)}\bigg)^2,\\
n_1 &:= \frac{-1}{A-B}\big(4\alpha \lambda (\lambda(A+B)+2MB\big)+\frac{16B(1-B)}{(1+B)^2}\big(\lambda^2(1+A)+\lambda M(1+B)\big),\\
r_1 &:=4\alpha^2+32\alpha\frac{B(1-B)}{(1+B)^3}-\bigg(\frac{\lambda^2(1-AB)+\lambda M (1-B^2)}{A-B}\bigg)^2.
\end{align*}

If
$(\ref{eqn:thm-2-janw-cc-1})$ holds, then $|n_1| \geq 2|m_1|$. Since $R_1$ is increasing, then $R_1(x) \geq m_1 + r_1 - |n_1|$,  which is nonnegative from $(\ref{eqn:thm-2-janw-cc-2})$. On the other hand, if
$(\ref{eqn:thm-2-janw-cc-3})$ holds, then $|n_1| < 2|m_1|$,  $R_1(x) \geq  (4 m_1 r_1 - n_1^2)/ 4 m_1$,  and
$(\ref{eqn:thm-2-janw-cc-4})$ implies $R_1(x) \geq 0$. Either case establishes $(\ref{eqn:thm-2-x})$.
\end{proof}

\begin{remark}
A graphical experiment using mathematica shows that $\RM(\mathtt{S}_\alpha(z))>0$ for all $\alpha \geq 0$ and $z \in \Delta$. But our computation restrict on $[0, 1/2] \cup[3/2, \infty)$.  Thus the result is open for $\alpha \in (0.5, 1.5)$.

\end{remark}

\section{Third order differential subordination for $\mathtt{S}_\alpha$}\label{section-3}
In this section we introduce an admissible class  $\Phi_{\mathtt{g}}[\Omega, q]$ as follows:
\begin{definition}\label{def-phi-g}
Let $\Omega$ be a set in $\mathbb{C}$ and $q \in \mathcal{Q}_0 \cap \mathcal{H}_0$. The class of admissible function $\Phi_{\mathtt{g}}[\Omega, q]$ consists of those functions $\phi : \mathbb{C}^4 \times \Delta \to \mathbb{C}$ that satisfy the following admissibility condition
\[ \phi(\beta_1, \beta_2, \beta_3, \beta_4; z) \notin \Omega\]
whenever
\begin{align*}
\beta_1=q(\zeta) \quad \quad \beta_2=\frac{m \zeta q'(\zeta)+(\alpha+1)q(\zeta)}{\alpha+1},
\end{align*}
\begin{align*}
\RM\left(\tfrac{4\alpha(\alpha+1)\beta_3+8\alpha(\alpha+1)\beta_2-(4\alpha^2+8\alpha+1)\beta_1}{2(\alpha+1)(\beta_2-\beta_1)}
+1\right)\geq m \RM\left(\tfrac{\zeta q''(\zeta)}{q'(\zeta)}+1\right),
\end{align*}
and
\begin{align*}
\RM\left(\tfrac{8\alpha(\alpha^2-1)\beta_4-4\alpha(\alpha+1)(6\alpha-1)\beta_3+
2(\alpha+1)(36\alpha^2-12\alpha-1)\beta_2+(40\alpha^3+16\alpha^2-18\alpha-6)\beta_1}
{2(\alpha+1)(\beta_2-\beta_1)}\right) \geq  m^2 \RM\left(\tfrac{\zeta^2 q''(\zeta)}{q'(\zeta)}\right),
\end{align*}
where $ z \in \Delta$, $ \alpha > 1$, $\zeta \in \partial\Delta\setminus E(q)$ and $m \geq 2$.
\end{definition}

Our  first result give the sufficient conditions  for the inclusion of $\mathtt{S}_\alpha$ in the admissible class $\Phi_{\mathtt{g}}[\Omega, q]$. In  this purpose, let define $\mathtt{g}_\alpha(z): = z \mathtt{S}_\alpha(z)$. Then a calculation along with \eqref{eqn:Bessel-struve-rec}, yields the recurrence relation
\begin{align}\label{eqn:rec-g}
z\, \mathtt{g}'_{\alpha}(z)=2\alpha \mathtt{g}_{\alpha-1}(z)+(1- 2\alpha) \mathtt{g}_{\alpha}(z),
\end{align}
which play the main role in this article. Now we will state and proof our main results on differential subordination involving $\mathtt{S}_\alpha$.

\begin{theorem}\label{theorem-1}
Let $\phi \in \Phi_{\mathtt{g}}[\Omega, q]$. Suppose that  $q \in \mathcal{Q}_0$ satisfy the following inequalities:
\begin{align}\label{hyp-thm-1}
\RM\left(\frac{\zeta q''(\zeta)}{q'(\zeta)}\right)\geq 0, \quad \left|\frac{\mathtt{g}_\alpha(z)}{q'(\zeta)}\right|\leq m.
\end{align}
For all $z \in \Delta$ and $\alpha >1$, if
\begin{align}
\{\phi(\mathtt{g}_{\alpha+1}(z), \mathtt{g}_\alpha(z), \mathtt{g}_{\alpha-1}(z), \mathtt{g}_{\alpha-2}(z); z\}\subset \Omega,
\end{align}
 then
$\mathtt{g}_{\alpha+1}(z) \prec q(z)$.
\end{theorem}
\begin{proof}
Define the analytic function $p$ in $\Delta$ as
\begin{align}\label{eqn-1-Thm1}
p(z):= \mathtt{g}_{\alpha+1}(z),  \quad \alpha>1.
\end{align}
Differentiate \eqref{eqn-1-Thm1} with respect to $z$. Then an application of \eqref{eqn:rec-g} for $\alpha+1$ yields
\begin{align}\label{eqn-2-Thm1}
\mathtt{g}_\alpha(z)=\frac{zp'(z)+(2\alpha+1)p(z)}{2(\alpha+1)}.
\end{align}
Differentiate both side of \eqref{eqn-2-Thm1} with respect to $z$ and then multiply by $z$ gives
\begin{align}\label{eqn-2-Thm11}
z\mathtt{g}'_\alpha(z)=\frac{z^2p''(z)+(2\alpha+1)zp'(z)}{2(\alpha+1)}.
\end{align}
Again use of \eqref{eqn:rec-g} implies
\begin{align}\label{eqn-3-Thm1}
\mathtt{g}_{\alpha-1}(z)=\frac{z^2p''(z)+4 \alpha zp'(z)+(4\alpha^2-1)p(z)}{4\alpha(\alpha+1)}.
\end{align}
Similarly, it can be shown that
\begin{align}\label{eqn-4-Thm1}
\mathtt{g}_{\alpha-2}(z)=\tfrac{z^3 p'''(z)+(6\alpha-1) z^2p''(z)+(12 \alpha^2 -8 \alpha-1) zp'(z)+(4 \alpha^2-1)(2\alpha-3)p(z)}{8\alpha(\alpha^2-1)}.
\end{align}

Now consider the four transformation $\beta_i: \mathbb{C}^4 \mapsto \mathbb{C}$, $i=1,2,3,4$, as follows:
\begin{align*}
(i)\; &\beta_1(r, s, t, u)=r\\
(ii)\;&\beta_2(r, s, t, u)=\frac{s+(\alpha+1)r}{\alpha+1}\\
(iii)\;&\beta_3(r, s, t, u)=\frac{t+4\alpha s+(4\alpha^2-1)r}{4\alpha(\alpha+1)}\\
(iv)\;&\beta_4(r, s, t, u) = \frac{u+(6\alpha-1)t+(12\alpha^2-8\alpha-1)s+(2\alpha-3)(4\alpha^2-1)r}{8\alpha(\alpha^2-1)}\hspace{1in}.
\end{align*}
Define $\psi:\mathbb{C}^4 \to \mathbb{C}$ as
\begin{align}\label{psi-phi}
\psi(r, s, t, u; z):=\phi(\beta_1, \beta_2, \beta_3, \beta_4; z).
\end{align}
Then clearly
 \[\psi(p(z), zp'(z), z^2p''(z), z^3 p'''(z); z)
=\phi(\mathtt{g}_{\alpha+1}(z), \mathtt{g}_{\alpha}(z), \mathtt{g}_{\alpha-1}(z), \mathtt{g}_{\alpha-2}(z); z).\]
From $(i)$-$(iv)$, it follows that
\begin{align*}
s&=2(\alpha+1)(\beta_2-\beta_1)\\
t&=4\alpha(\alpha+1)\beta_3+8\alpha(\alpha+1)\beta_2-(4\alpha^2+8\alpha+1)\beta_1\\
u&=8\alpha(\alpha^2-1)\beta_4-4\alpha(\alpha+1)(6\alpha-1)\beta_3+
2(\alpha+1)(36\alpha^2-12\alpha-1)\beta_2\\
&\quad \quad+(40\alpha^3+16\alpha^2-18\alpha-6)\beta_1.
\end{align*}
Thus the admissibility for $\phi \in \Psi_g[\Omega,q]$ as stated in the  Definition \ref{def-phi-g} is
equivalent to the  admissible condition  for $\psi \in \Psi_n[\Omega,q]$, $n=2$ as given in Definition \ref{def-miller-anton}.
It is evident that the result follows from Theorem $\ref{thm-miller-anton}$  provided the hypothesis \eqref{hyp-thm-1} hold.
\end{proof}

Now consider the case,  when $\Omega \neq \mathbb{C}$ is a simple connected domain, then for  some conformal mapping $h$ of $ \Delta$ to $\Omega$, we have $\Omega=h(\Delta)$.
In this case the class $\Psi_g[h(\Delta), q]$ is denoted as $\Psi_g[h, q]$  and the following result is an immediate consequence of Theorem $\ref{theorem-1}$.

\begin{theorem}\label{theorem-2}
Let $\phi \in \Phi_{\mathtt{g}}[\Omega, q]$. Suppose that  $q \in \mathcal{Q}_0$ satisfy the hypothesis \eqref{hyp-thm-1}.  For all $z \in \Delta$ and $\alpha >1$, if
\begin{align}\label{hyp2-thm-3}
\{\phi(\mathtt{g}_{\alpha+1}(z), \mathtt{g}_\alpha(z), \mathtt{g}_{\alpha-1}(z), \mathtt{g}_{\alpha-2}(z); z\}\prec h(z),
\end{align}
 then $\mathtt{g}_{\alpha+1}(z) \prec q(z)$.
\end{theorem}

If the behaviour of $q$ on $\partial \Delta$ is not known, then Thorem \ref{theorem-1} can be extended as in the following result.

\begin{theorem}\label{theorem-3}
Let $\Omega \in \mathbb{C}$ and $q$ be univalent in $\Delta$ with $q(0)=0$ .  Suppose that   $\phi \in \Phi_{\mathtt{g}}[\Omega, q_r]$ for some $r \in (0,1)$, where $q_r(z)=q(rz)$
satisfy
\begin{align}\label{hyp-thm-3}
\RM\left(\frac{\zeta q_r''(\zeta)}{q_r'(\zeta)}\right)\geq 0, \quad \left|\frac{\mathtt{g}_\alpha(\zeta)}{q_r'(\zeta)}\right|\leq m.
\end{align}
For all $z \in \Delta$ and $\alpha >1$, if
\begin{align}
\phi(\mathtt{g}_{\alpha+1}(z), \mathtt{g}_\alpha(z), \mathtt{g}_{\alpha-1}(z), \mathtt{g}_{\alpha-2}(z); z) \in \Omega
\end{align}
 then $\mathtt{g}_{\alpha+1}(z) \prec q(z)$.
\end{theorem}
\begin{proof}
It follows from Theorem $\ref{theorem-1}$ that $\mathtt{g}_{\alpha+1}(z) \prec q_r(z)$. Now the result can be deduced from the fact that $q_r(z) \prec q(z)$ for all fixed $r \in (0,1)$ and $z \in \Delta$.
\end{proof}

Our next result yields the best dominant of the differential subordination  \eqref{hyp2-thm-3}.
\begin{theorem}
Let $h$ be  univalent in $\Delta$, and let $\phi: \mathbb{C}^4 \times \Delta \to \mathbb{C}$ and $\Psi$ be given by  \eqref{psi-phi}. Suppose that
the differential equation
\begin{align}\label{hyp-thm-4}
\Psi( q(z), z q'(z), z^2 q''(z), z^3 q''(z); z)=h(z),
\end{align}
has a solution $q(z)$ with $q(0)=0$ and satisfies the condition \eqref{hyp-thm-1}.\\
 If $\phi \in \Phi_g[h, q_r]$ and   $\phi( g_{\alpha+1}(z), g_\alpha(z), g_{\alpha-1}(z), g_{\alpha-2}(z); z)$ is analytic
in $\Delta$,  then  \eqref{hyp2-thm-3} implies that  $\mathtt{g}_{\alpha+1}(z) \prec q(z)$, and $q$ is the best dominant.
\end{theorem}
\begin{proof}
From Theorem $\ref{theorem-1}$, it is evident that $q$ is a dominant of \eqref{hyp2-thm-3}.  Since $q$ satisfies  \eqref{hyp-thm-4}, it is also a  solution of \eqref{hyp2-thm-3} and therefore $q$ is dominated by all dominant.  This implies $q$ is the best dominant.
\end{proof}

\end{document}